\documentclass{gtpart}     
\usepackage{epsfig,color,graphicx}

\title{On the genericity of loxodromic actions}

\author{Bert Wiest}
\givenname{Bert}
\surname{Wiest}
\address{Bert Wiest, UFR Math\'ematiques, Universit\'e de Rennes 1, 35042 Rennes Cedex, France}
\email{bertold.wiest@univ-rennes1.fr}
\urladdr{http://perso.univ-rennes1.fr/bertold.wiest}


\volumenumber{}
\issuenumber{}
\publicationyear{}
\papernumber{}
\startpage{}
\endpage{}
\doi{}
\MR{}
\Zbl{}
\received{}
\revised{}
\accepted{}
\published{}
\publishedonline{}
\proposed{}
\seconded{}
\corresponding{}
\editor{}
\version{}

\makeautorefname{notation}{Notation}%

\newtheorem{theorem}{Theorem}[section]
\newtheorem{lemma}[theorem]{Lemma} 
\newtheorem{proposition}[theorem]{Proposition}

\theoremstyle{definition}
\newtheorem{definition}[theorem]{Definition}    
\newtheorem{remark}[theorem]{Remark}

\newtheorem{openproblems}[theorem]{Open Problems}
\newtheorem{example}[theorem]{Example}

\newtheorem{notation}[theorem]{Notation}

\def\N{{\mathbb N}}

\renewcommand{\phi}{\varphi}
\def\marg#1{}
\begin{document}

\begin{abstract}
Suppose that a finitely generated group~$G$ acts by isometries on a $\delta$-hyperbolic space, with at least one element acting loxodromically. Suppose that the elements of~$G$ have a normal form such that the language of normal forms can be recognized by a finite state automaton. Suppose also that a certain compatibility condition linking the automatic and the $\delta$-hyperbolic structures is satisfied. Then we prove that in the ``ball'' consisting of elements of~$G$ whose normal form is of length at most~$l$, the proportion of  elements which act loxodromically is bounded away from zero, as $l$ tends to infinity. 
We present several applications of this result, including the genericity of pseudo-Anosov braids.
\end{abstract}

\maketitle


\section{Introduction} 

For several decades, it was an unproven, but widely believed slogan that ``generic elements of the mapping class groups of surfaces are pseudo-Anosov''. Here the word ``generic'' admits several different interpretations. For one of the possible interpretations, this slogan has recently been proven in a very satisfactory way. 
Indeed, A.~Sisto~\cite{SistoGeneric}, inspired by previous work of J.~Maher \cite{Maher}, proved that the element obtained by a long random walk in the Cayley graph of the mapping class group of a surface tends to be pseudo-Anosov, with a probability which tends to~$1$ exponentially quickly as the length of the walk tends to infinity. (See also~\cite{Rivin} for a completely different approach.)
Actually, Sisto works in a far more general context: if a finitely generated, non-elementary group acts on a $\delta$-hyperbolic space, then the probability that the element obtained by a random walk in the group acts non-loxodromically 
decays exponentially quickly with the length of the walk. In the special case of the mapping class group acting on the curve complex of the surface, this implies the original slogan.

In the present paper, we use a different notion of genericity. For us, a generic element is not obtained by a long random walk, but it is picked randomly with uniform probability among the elements in a large ``ball'' centered on the identity element in the Cayley graph of the group. Almost nothing was known about the genericity (in this sense) of pseudo-Anosov elements until the paper \cite{CarusoGeneric1} by S.~Caruso.

Caruso's result states that in the braid group~$B_n$, equipped with Garside's generators, the proportion of elements in the ball of radius~$l>1$ in the Cayley graph which are \emph{rigid} and pseudo-Anosov is bounded below by a strictly positive constant (independently of~$l$). The word ``rigid'' will be defined below -- for the moment we just point out that the proportion of rigid elements does definitely \emph{not} tend to~$1$ as $l$~tends to infinity. 

In the present paper we generalize Caruso's result to a much larger framework. Here is a rough statement of our main result -- for a detailed and more general statement see Theorem~\ref{T:main} and Proposition~\ref{P:rigid}.
\smallskip

{\bf Theorem} \ \,{\sl Consider a finitely generated group~$G$ which acts on a $\delta$-hyperbolic space~$X$, with at least one element acting loxodromically. Suppose that~$G$ satisfies a certain weak automaticity condition, and also a certain compatibility condition linking the automatic and the $\delta$-hyperbolic structures. Then, among the elements of~$G$ whose automatic normal form is of length at most~$l$, the proportion of those which are rigid and act loxodromically is bounded below by a positive constant. Moreover, these rigid loxodromic elements have the property that their axis in~$X$ passes uniformly close to a pre-defined base point of~$X$.}
\smallskip

We present several applications of this result. 
\marg{Referee: The author should also state explicitly whether or not the results
hold for Gromov hyperbolic groups, and compare and contrast his
approach with that of Calegari and Fujiwara (for example in "Combable
functions, quasimorphisms, and the central limit theorem" and "Stable
commutator length in word hyperbolic groups") who also make use of
automatic structures on groups.}
Firstly, we show that a positive proportion of elements of~$\mathrm{PSL}(2,\mathbb Z)$ act on the Farey graph in a loxodromic fashion -- this is already well known, but a nice illustration of our proof techniques. 
Secondly, we show that 
the above result holds for word-hyperbolic groups, acting on their own Cayley graphs.
Thirdly, we explain why there is hope that Theorem~\ref{T:main} can be applied in the realm of relatively hyperbolic groups, acting on their coned-off Cayley graphs. 
Fourthly, we study the action of the braid group~$B_n$ on the so-called  ``additional length complex''. We prove that generic elements act in a loxodromic manner, and hence are pseudo-Anosov. We also re-prove Caruso's result. 
Finally, we look at the fundamental groups of compact, non-positively curved cubical complexes and their action on the associated contact graphs. We prove that, under mild hypotheses, a positive proportion of elements acts loxodromically.


\section{Statement of the main result}

Throughout this paper, we will be interested in the properties of generic elements of a monoid~$G$. We always suppose that each element $g$ of~$G$ has a preferred representative word $g=g_1\cdot\ldots\cdot g_k$, called the \emph{normal form} of the element. The letters $g_i$ must belong to some finite set~$\mathcal S$ of arrows of a category~$\hat G$ which contains the monoid~$G$ as a sub-category: $G\subset \hat G$. (We recall that a monoid is a category with only one object.) 

We also suppose that the language of normal forms of elements of~$G$ is recognized by a deterministic finite state automaton (FSA).
We recall briefly what this means: we have a finite graph, whose vertices are called the \emph{states} of the automaton, and whose edges are oriented and labelled by elements of~$\mathcal S$. From each state there is exactly one exiting edge for every element of~$\mathcal S$. 
There is a unique \emph{start state} -- we shall assume that the start state is not the target of any arrow. There's also a unique \emph{fail state} -- all edges exiting the fail state immediately loop back to it. 
Some of the states (possibly all except the fail state) are labelled as \emph{accept states}.
A word over the alphabet~$\mathcal S$ is \emph{recognized} or \emph{accepted} by the automaton if the path in the automaton beginning at the start state and following the edge labels of the word ends in an accept state. 
For an introduction to the theory of finite state automata, the reader can consult \cite[Chapter~1]{ECHLPT} or~\cite{NeumannShapiro}.

\begin{remark}
(1) For many examples we are interested in, 
the equality $\hat G=G$ holds. Moreover, often $G$ and $\hat G$ are actually \emph{groups}. 
In that case, the normal form is simply a unique preferred writing of each element of the group~$G$ as a product of the generators $\mathcal S$ of~$G$. However, our more general setup is often useful, and comes at no extra cost. 

(2) It may happen that there are pairs of elements $g_1, g_2$ of $\mathcal S$ such that $g_1g_2$ is not well-defined in~$\hat G$ -- this occurs if the category~$\hat G$ has more than one object and if the target object of~$g_1$ does not coincide with the source object of~$g_2$.  However, such a string $g_1g_2$ cannot appear in any word recognized by our FSA --  it is necessarily sent to the fail state.
\end{remark}

We say a word~$w$ over the alphabet~$\mathcal S$ is \emph{rigid} if the word $w^n$ (the word repeated $n$ times) is accepted by the automaton for all positive integers~$n$.
For instance, this is the case if the words $w$ and $w^2$ are both accepted, and if the paths traced out in the automaton by the two words end in the same state. We say an element $g$ of~$G$ is \emph{rigid} if the normal form word representing it is rigid, i.e.\ if the normal form of~$g^n$ is equal to the normal form of~$g$, repeated $n$~times. In the past, rigid elements have only been defined for Garside groups \cite{BGGMrigid}, but the preceding definition is just an obvious generalization.

If $A$ is the transition matrix of the automaton (ignoring the start- and fail state), then the $(i,j)$th entry of $A^l$ counts the number of words of length~$l$ starting at the $i$th and ending at the $j$th state. This makes it possible to use Perron-Frobenius theory to asymptotically count words~\cite{BrouwerHaemers,Varga}. 

We say a finite state automaton is \emph{recurrent} if there is some integer~$l$ such that there is a path of length exactly~$l$ from from any state (other than the start state or fail state) to itself and to any other of those states. In terms of matrices, this means that $A^l$ is a matrix with all entries strictly positive -- this property is very useful in Perron-Frobenius theory.
The automata which we are interested in in practice (related to mapping class groups, for instance) are unfortunately not always recurrent; nevertheless the reader should keep the recurrent case in mind, because it is simpler while already containing the key of the argument.

For the purposes of this paper we introduce another notation (which is unnecessary if the automaton under consideration is recurrent). We consider the set of states which can be reached from all non-fail states of the automaton. This set of states, together with the edges connecting them, will be called the \emph{accessible sub-automaton}, and its states the \emph{accessible states}.
Note that the accessible sub-automaton may be empty, but when it is not, then paths can enter, but never leave it. 
We shall say a FSA is \emph{dominated} by its accessible sub-automaton if the sub-automaton is non-empty, recurrent, 
and if the exponential growth rate of the number of words readable in its accessible sub-automaton is strictly larger than that of the original automaton, with the sub-automaton removed. 

\begin{definition}
We say that $G$~satisfies the \emph{automatic normal form hypothesis} if every element of~$G$ has a unique preferred word over the alphabet~$\mathcal S$ representing it, and if the language of preferred representatives (normal forms) is recognized by a deterministic finite state automaton (FSA) with the following properties~: 
\begin{itemize}
\item the automaton must be dominated by its accessible sub-automaton, and
\item there is at least one word $w_{\rm rigid}$ such that both $w_{\rm rigid}$ and $w_{\rm rigid}^2$ are accepted by the automaton, where the paths traced out in the automaton by the two words both end in the same accept state~$E$. Moreover, this state~$E$ has to lie in the accessible sub-automaton.
\end{itemize} 
\end{definition}

Note that the automatic normal form hypothesis does not require any kind of the fellow traveller conditions familiar from automatic or combable groups. It is really a fairly mild hypothesis, as the following examples illustrate. 

\begin{example}\label{E:RecurrentANFH}
The automatic normal form hypothesis is satisfied for instance if the language of normal forms is recognized by a FSA which is recurrent and which moreover contains an accept state~$V$ and a letter $g\in\mathcal S$ such that the two arrows labelled~$g$ exiting from $V$ and from the start state end in the same, non-fail state. Indeed, under these hypotheses, the accessible sub-automaton is the whole automaton (except start- and fail state), and any word traced out in the FSA which starts with the letter~$g$ and ends in the state~$V$, can play the role of the word $w_{\rm rigid}$. 
\end{example}

\begin{example}\label{E:B_n^+}
We consider $G=B_n^+$, the positive braid monoid, equipped with Garside's generating set~$\mathcal S$ of nontrivial positive permutation braids, which has $n!-1$ elements. We set $\hat G=G$, and construct a FSA which has one start state, one fail state, and $n!-1$ other states, which are labelled by the $n!-1$ generators, and which are all accept states. In this automaton, every edge is labelled by one of the $n!-1$ non-trivial positive permutation braids, and an edge labelled $g$ either terminates in the state labelled $g$ or in the fail state. The automaton is constructed so as to encode the left-weighting condition of the Garside normal form, see \cite{ElrifaiMorton}.
The accessible sub-automaton of this automaton contains $n!-2$ states, corresponding to all positive permutation braids other than $1$ and~$\Delta$.
As proven in~\cite{CarusoGeneric1}, it is recurrent with~$l=5$. The normal form of any rigid braid can play the role of the word~$w_{\rm rigid}$. Thus the automatic normal form hypothesis is satisfied in this case.
\end{example}


\begin{example}\label{E:pi_1(CAT(0))}
Let $Y$ be a non-positively curved (locally CAT(0)) cubical complex, and let $b$ be one of its vertices. Let $G=\pi_1(Y,b)$ be its fundamental group, with base point in~$b$. Let $\hat G$ be the fundamental groupoid of $Y$, where the set of base points equals the set of vertices of~$Y$ -- so $G$ is a sub-group of the groupoid~$\hat G$. Let $\mathcal S$ be the set of all oriented diagonal paths across all cubes; thus every $n$-dimensional cube in~$Y$ gives rise to $2^n$ elements of~$\mathcal S$. 
It was shown by Niblo and Reeves~\cite{NibloReeves} that the fundamental groupoid (and hence also the fundamental group~$G$) is bi-automatic. Their normal form for elements of~$\hat G$ uses the alphabet~$\mathcal S$, and is left-greedy: every hyperplane is traversed as early as possible. This normal form can be recognized by a FSA, which has a start state, one state for every element of~$\mathcal S$, and a fail state. 
An edge labelled $g\in\mathcal S$ can lead either to the state labelled~$g$, or to the fail state. States labelled by paths which end in the base point~$b$ are accept states. If this FSA happens to be recurrent, then the automatic normal form hypothesis is satisfied, by the criterion in Example~\ref{E:RecurrentANFH}.\marg{I hope that argument is detailed enough.}

Depending on the nature of the complex~$Y$, the FSA may or may not be  recurrent. For instance, if $Y$ is the standard torus constructed from one vertex, two edges and one square, then the associated FSA is \emph{not} recurrent: a horizontal edge path can only be followed by another horizontal edge path in the same direction, and nothing else. By contrast, Figures~2 and~3 of \cite{CrispWiest1} give decompositions of closed surfaces~$S$, with $\chi(S)<0$, as non-positively curved cube complexes; the finite state automata associated to these cubings can be checked to be recurrent. This concludes Example~\ref{E:pi_1(CAT(0))}.
\end{example}

\begin{notation}
For an integer $l$, we denote $B(l)$ be the set of elements of $G$ whose normal form is of length at most~$l$. Warning: since the normal form words need not be geodesics in the Cayley graph, this set does not necessarily coincide with the ball of radius~$l$ and center $1_G$ in the Cayley graph of~$G$.
\end{notation}
\medskip

Now we suppose\marg{I think it's not necessary to explain what it means for a category to act on the left on a set.} that~$\hat G$ acts on the left by isometries on a $\delta$-hyperbolic complex~$X$ (which need not be proper). Let~$P$ be a point of~$X$ -- we think of it as a base point, and it is fixed once and for all. Then we define $c=\max_{g\in\mathcal S}d_X(P,g.P)$. (It is philosophically satisfying, but of no mathematical importance, to choose~$P$ so as to make~$c$ as small as possible.) Now if $x_1x_2x_3\ldots$ is a word in the alphabet~$\mathcal S$ which is well-defined in~$\hat G$, then any two successive points in the sequence $P$, $x_1.P$, $x_1x_2.P$, $x_1x_2x_3.P, \ldots$ in~$X$ are at distance at most~$c$. 

We shall need a compatibility condition between the automatic and the hyperbolic structures: 

\begin{definition}
Suppose a category~$\hat G$ acts on on a $\delta$-hyperbolic complex $(X,P)$, and that the sub-monoid $G\subset \hat G$ satisfies the automatic normal form hypothesis.
\marg{Don't really need the full automatic normal form hypothesis in the definition -- it would suffice to talk about a family of words over an alphabet $\mathcal S$ which acts on $X$. But it's easier to write this way.} 
We say that the \emph{geodesic words hypothesis} is satisfied if there exists a constant $R\geqslant 0$ such that for every  
normal form word $x_1x_2x_3\ldots$, and for every $i,j\in\N$ with $0\leqslant i<j$, the sequence of points 
$$x_1x_2\ldots x_i.P, \ \ x_1x_2\ldots x_{i+1}.P, \ \ldots\ldots \ , x_1x_2\ldots x_j.P$$
stays in the $R$-neighbourhood of every geodesic connecting its end-points. 
\end{definition}

\begin{remark}
(a) The geodesic words hypothesis is actually equivalent\marg{Reference for this equivalence?} to the condition that the family of paths
$$\{(P, \ x_1.P, \ x_1x_2.P, \ x_1x_2x_3.P,  \dots) | \ x_1x_2x_3\ldots \textrm{ a word recognized by the FSA} \}$$
forms a uniform family of unparameterized quasi-geodesics in $X$. However, we shall not use this equivalent point of view. 

(b) The geodesic words hypothesis is a \emph{shadowing} condition: we require that the ``shadows'' of normal form words to~$X$ are unparametrized quasi-geodesics. For work on related conditions, in the context of mapping class groups and $Out(F_n)$, see~\cite{LRT}, and the references cited in the introduction of that paper.
\end{remark}

We recall the classification of isometries of $\delta$-hyperbolic spaces: an isometry~$\phi$ of~$X$ acts \emph{elliptically} if the orbit $P$, $\phi(P)$, $\phi(\phi(P))$, $\phi(\phi(\phi(P))),\ldots$ in~$X$ is bounded. It acts \emph{parabolically} if the orbit is unbounded but $\lim_{n\to\infty}\frac{1}{n}\cdot d_X(P, \phi^n(P))= 0$. In this case, the sequence of points $P$, $\phi(P)$, $\phi^2(P),\ldots$ of $X$, is \emph{not} a quasi-geodesic in~$X$: the point $\phi^n(P)$ gets arbitrarily far away from any geodesic connecting $P$ and~$\phi^{2n}(P)$; moreover,  the isometry of $X$ given by the action of~$\phi$ has a unique fixed point on~$\partial X$.  
Finally, $\phi$ acts \emph{loxodromically} (or \emph{hyperbolically}) if $\lim_{n\to\infty}\frac{1}{n}\cdot d_X(P, \phi^n(P))$ exists and is positive. In this case, the orbit of~$P$ under $\langle \phi\rangle$ is a quasi-geodesic, and the isometry of $X$ given by the action of $\phi$ has two fixed points on~$\partial X$. For a proof that any element belongs to exactly on of these three classes see \cite[Section 9]{CDP} (which requires $X$ to be proper, but the proof does not use this hypothesis). See also~\cite[Section 3]{CCMT}. 

\begin{notation}
Suppose a category~$\hat G$ acts on a $\delta$-hyperbolic space, and that the sub-monoid~$G$ satisfies the automatic normal form hypothesis. We say that \emph{there are interesting loxodromic actions} if  there exists a loop in the accessible sub-automaton such that the product of elements of~$\mathcal S$ read along this loop acts loxodromically on $X$. This is for instance the case if there are rigid elements which act loxodromically.
\end{notation}

We are now ready to state our main result.

\begin{theorem}\label{T:main}
Let $G$ be a monoid and $\hat G$ a category with $G\subset \hat G$, and with~$\hat G$ acting on a $\delta$-hyperbolic complex~$X$. Suppose that the automatic normal form hypothesis and the geodesic words hypothesis are satisfied. Suppose further that there are interesting loxodromic actions. Then
$$
\liminf_{l\to\infty} \frac{\#\{\textrm{elements of }B(l)\textrm{ which are rigid and act loxodromically on }X\}}{\#B(l)} > 0.
$$
\end{theorem}

We remark that the proportion of rigid elements does definitely \emph{not} tend to~1 as $l$ tends to infinity -- in this sense, Theorem~\ref{T:main} cannot be improved upon. Being rigid is a very interesting property, in that it has remarkable dynamical consequences. Very roughly speaking, for a rigid element, the guts of the action happen close to the base point:

\begin{proposition}\label{P:rigid}
Suppose the hypotheses of Theorem~\ref{T:main} are satisfied.
\marg{Again, I don't really need the full automatic normal form hypothesis.} 
Let $P$ be the base point of~$X$, and let $R$ be the constant appearing in the geodesic words hypothesis. Let $g$ be a rigid element of~$G$. 
\begin{itemize}
\item $g$ cannot act parabolically on~$X$.
\item If $g$ acts elliptically, then $d_X(P,g^n.P)\leqslant 3R$ for all~$n$. 
\item If $g$ acts loxodromically, and if $a\subset X$ is a geodesic axis, then the base point~$P$ is at distance at most~$R$ from~$a$.
\end{itemize}
\end{proposition}

\begin{openproblems} 
\begin{enumerate}
\item It should be possible to use Theorem~\ref{T:main} in order to study the proportion of pseudo-Anosov elements in all mapping class groups. We recall that a mapping class is pseudo-Anosov if and only if it acts loxodromically on the curve complex of the surface. It seems likely that any mapping class group, equipped with any reasonable normal form recognized by a FSA, satisfies the geodesic words hypothesis. One tool for checking the geodesic words hypothesis could be the fact~\cite{MM4} that train track splitting sequences give rise to unparametrized quasi-geodesics in the curve complex -- see also~\cite{MasMoshSchl}. 
\item Apply the theorem in other contexts, for instance of $Out(F_n)$ acting on the free factor complex $\mathcal{FF}_n$. We recall that in this setup, an element acts loxodromically if and only if it is iwip. 
\item Prove the stronger result that the proportion of elements of~$G$ which act loxodromically actually tends to~$1$ (we have only proven that its $\liminf$ is positive). One possible strategy for doing so involves constructing a ``blocking element'' in~$G$, analogue to the blocking braids of~\cite{CarusoWiestGeneric2}.

\item Strengthen the statements by proving that the translation distance in~$X$ of the action of generic elements with normal forms of length~$l$ is bounded below by a linear function in~$l$ (c.f.\ \cite{MaherLin}).

\item Another possible strengthening would be the genericity not only of loxodromic, but even of  contracting actions, in the sense of Sisto~\cite{SistoGeneric}.
\end{enumerate}
\end{openproblems}


\section{Proofs}

\begin{proof}[Proof of Theorem~\ref{T:main}]
Throughout the proof, our strategy for counting \emph{elements} of~$G$ with various properties will be to count \emph{normal form words} representing elements with the same properties.

The first step of the argument is to prove that  the proportion of elements which are rigid is bounded away from zero. 

\begin{lemma}\label{L:RigidNotRare}
$$
\liminf_{l\to\infty} \frac{\#\{\textrm{rigid words of length }l\textrm{ recognized by the FSA}\}}{\phantom{rigi}\#\{\textrm{words of length }l\textrm{ recognized by the FSA}\}} > 0.$$
\end{lemma}

\begin{proof} We recall that there is, by hypothesis, a path in the automaton, starting at the start state and ending at some accept state $E$ in the accessible sub-automaton, tracing out a word $w_{\rm rigid}$ in the automaton such that the path along the word~$w_{\rm rigid}^2$ ends in the same state~$E$. In particular, the word $w_{\rm rigid}$ is rigid.
\marg{I think this explanation is sufficiently detailed.} 
Since the automaton is dominated by the accessible sub-automaton, we have
$$
\liminf_{l\to\infty} \frac{\#\{\textrm{words of length }l\textrm{ with prefix }w_{\rm rigid},\textrm{ recognized by the FSA}\}}{\phantom{}\#\{\textrm{words of length }l\textrm{ recognized by the FSA}\}} > 0.$$

Next we observe that, for any word~$w_\circlearrowleft$ read out along a loop starting and ending in the state~$E$, the word $w_{\rm rigid}w_\circlearrowleft$ is also rigid. Let us denote the by~$k$ the length of the word~$w_{\rm rigid}$. Now, among the paths of length $l-k$ in the automaton (in fact in the accessible sub-automaton) which start at the state~$E$, the proportion of those which also end at~$E$ tends to some positive number as $l$~tends to~$\infty$ -- for a proof of this fact, see for instance \cite[Lemma 3.5(i) and (ii)]{CarusoGeneric1}.

In summary, among the words of length~$l$, there is a positive proportion which starts with the word~$w_{\rm rigid}$, and among those, there is a positive proportion of words which end at the state $E$. Since all these words are rigid, the lemma is proven.
\end{proof}

We now turn to the second step of the proof of Theorem~\ref{T:main}, namely proving that among the rigid elements constructed in the previous step, the proportion of loxodromically acting ones tends to 1. (Actually, the convergence happens exponentially quickly, but we will not use or prove this.)  

Let $w$ be a word in normal form which is rigid -- this means that the word $w^n$ is also in normal form for all positive integers~$n$, and thus that the geodesic words hypothesis applies to all~$w^n$. We will prove that this imposes serious restrictions on how the element of~$G$ represented by~$w$ can act on~$X$.

\begin{lemma}\label{L:NonPaObvious}
Let~$w$ be a rigid word in normal form. Then $w$ cannot act parabolically; moreover, if it acts elliptically, then it does not contain any subword whose action moves the base point~$P$ by more than~$5R$. (Here $R$ is the constant appearing in the geodesic words hypothesis.) 
\end{lemma}

\begin{proof}
The rigid element represented by $w$ cannot act \emph{parabolically} on~$X$, because if it did, then the infinite sequence of points $P$, $w.P$, $ww.P$, $www.P, \ldots$ would not lie uniformly close to a geodesic ray, as required by the geodesic words hypothesis. 

Next, we look at the case where the rigid element $w$ acts \emph{elliptically}. We denote $D=d_X(P,w.P)$. Since $w$ acts by isometries, we have $d_X(w^k.P,w^{k+1}.P)=D$ for all integers~$k$. 

{\bf Claim A } We claim that $D\leqslant 3R$. 

In order to prove Claim A, we suppose, for a contradiction, that $D=3R+\epsilon$, for some positive integer~$\epsilon$. Since $w$ acts elliptically, the orbit $P$, $w.P$, $ww.P$, $www.P, \ldots$ is bounded -- let $\rho$ be its diameter. Now we choose an integer~$N$ so large that $\frac{\rho}{N}<\epsilon$. 
By the geodesic words hypothesis, the geodesic from~$P$ to~$w^N.P$ contains all points $w^i.P$, for $0\leqslant i\leqslant N$, in its $R$-neighbourhood. In particular, for every point $w^i.P$ there is a point $v_i$ on this geodesic with $d_X(w^i.P, v_i)\leqslant R$. We now have the~$N$ points $\nu_1,\ldots,\nu_N$ on a geodesic of length at most~$\rho$, so there is a pair of indices $(i,j)$ with $0\leqslant i<j\leqslant N$ such that $d_X(v_i, v_j)\leqslant \frac{\rho}{N}<\epsilon$. This implies that $d_X(w^i.P, w^j.P)<2R+\epsilon$. Hence the point $w^{i+1}.P$, which lies at distance exactly $3R+\epsilon$ from $w^i.P$, cannot lie in the $R$-neighbourhood of a geodesic segment from $w^i.P$ to $w^j.P$. This contradicts the geodesic words hypothesis, and the claim is proven.

We have just seen that a rigid element~$w$ which acts elliptically cannot move the base point very far: $d_X(P,w.P)\leqslant 3R$. In fact, neither can any subword of~$w$: 

{\bf Claim B } If $w=x_1x_2x_3\ldots$ is rigid, we claim that any subword $x_i x_{i+1}\ldots x_{j-1} x_j$ satisfies
$$
d_X(P, x_i x_{i+1}\ldots x_{j-1} x_j.P)\leqslant 5R
$$

In order to prove this claim, we observe that $x_1\ldots x_{i-1}.P$ and $x_1\ldots x_j.P$ each lie $R$-close to some point of a geodesic from~$P$ to~$w.P$. Since this geodesic is only of length at most $3R$, we obtain
$$
d_X(x_1\ldots x_{i-1}.P, x_1\ldots x_j.P)\leqslant R+3R+R=5R
$$
Since $x_1\ldots x_{i-1}$ acts by an isometry on~$X$, we deduce that $d_X(P, x_i\ldots x_j.P)\leqslant 5R$, also. The proof of Claim B and of Lemma~\ref{L:NonPaObvious} is complete.
\end{proof}

We observe that words whose action displaces the base point by more than $5R$ do actually exist (by the hypothesis that there are \emph{loops} in the accessible sub-automaton which act loxodromically). Let $w_{\rm far}$ be one such word.

Now we recall that we have a distinguished rigid word $w_{\rm rigid}$ recognized by the automaton, whose length we denote~$k$ and which is read in the automaton along a path which ends in an accessible accept state which we denote~$E$. 
As seen in the proof of Lemma~\ref{L:RigidNotRare}, there is a non-negligible proportion of elements of length~$l$ which are recognized by the automaton and which are of the form $w_{\rm rigid}w_\circlearrowleft$, where the word~$w_\circlearrowleft$ is read along a loop of length $l-k$ based at the state~$E$. In particular, these words $w_{\rm rigid}w_\circlearrowleft$ are rigid. Now, we have seen in Lemma~\ref{L:NonPaObvious} that these words are guaranteed to act loxodromically if they contain a subword $w_{\rm far}$ which moves the base point by more than $5R$. 

Among the words of length~$l-k$ which can be read along loops based at~$E$ in the accessible sub-automaton,  the proportion of those that do not contain the subword~$w_{\rm far}$ tends to~$0$ as $l$ tends to infinity. For a proof of this statement, see Lemma 3.5(iii) of~\cite{CarusoGeneric1}. (In fact, this convergence is exponentially fast.) 

In summary, the proportion of rigid, loxodromically acting words among all recognized words of length equal to~$l$ is bounded away from~$0$. The very last step is to deduce that the same holds among all recognized words of length \emph{less than or} equal to~$l$. The proof of Theorem~\ref{T:main} is now complete.
\end{proof}

\begin{remark}\label{R:SphereBall}
For later reference, we point out that we started by proving the analogue of Theorem~\ref{T:main} for the \emph{sphere} consisting of elements whose normal form is of length equal to~$l$, and deduced the result for the \emph{balls} $B(l)$ from that.
\end{remark}

\begin{proof}[Proof of Proposition~\ref{P:rigid}]
All the work for this proof has already be done in the proof of Theorem~\ref{T:main}. Lemma~\ref{L:NonPaObvious} tells us that~$g$ cannot act parabolically, and elliptically acting elements cannot move the base point by more than~$5R$. Claim~A gives the stronger bound of~$3R$. For the loxodromic case, the geodesic words hypothesis implies that, for all positive integers~$k$, the point $g^k.P$ lies $R$-close to any geodesic from $P$ to~$g^{2k}.P$.%
\marg{Is it clear that this implies that $P$ is $R$-close to the axis?} 
Thus $P$ lies $R$-close to any geodesic from $g^{-k}.P$ to~$g^k.P$. 
\end{proof}


\section{Examples}

In this section we study four examples: the group $PSL(2,\mathbb Z)$ acting on the Farey graph, braid groups acting on the additional length complex, relatively hyperbolic groups acting on their associated $\delta$-hyperbolic space, and of  fundamental groups of compact, locally CAT(0) cubical complexes acting on the contact graph.


\subsection{$\mathrm{PSL}(2,\mathbb Z)$ acting on the Farey graph}

As a warm-up exercise, we will look at the group $\mathrm{PSL}(2,\mathbb Z)$, acting on the Farey graph. One motivation for this study is that $\mathrm{SL}(2,\mathbb Z)$ is the mapping class group of the once-punctured torus, which acts on its curve complex, the Farey graph -- see \cite[Section 1.5]{MM2}. Saying that a matrix $M$ in $\mathrm{SL}(2,\mathbb Z)$ acts loxodromically is equivalent to saying that~$M$, seen as a homeomorphism of the once-punctured torus, is of pseudo-Anosov type.

In this framework, Theorem~\ref{T:main} does not contain any new information, since the full genericity (in the sense of this paper) of pseudo-Anosov elements has been known for a long time. A proof of this fact is in~\cite[Theorem 18]{Rivin}, and the closely related case of the 3-strand braid group is treated very explicitly in~\cite{AK}. However, our aim here is to illustrate the proof of Theorem~\ref{T:main}

The group $\mathrm{PSL}(2,\mathbb Z)$ has a presentation
$$\langle \ A,B \ | \ ABA=BAB, (ABA)^2=1\ \rangle, \text{ \ where \ } A=\begin{pmatrix}
1 & 1 \\
0 & 1 \\
\end{pmatrix}, \ \ \ 
B=\begin{pmatrix}
1 & 0 \\
-1 & 1 \\
\end{pmatrix}.$$
We shall outline the proof of the following special case of Theorem~\ref{T:main}:

\begin{proposition}
Let $B(l)\subset \mathrm{PSL}(2,\mathbb Z)$ be the vertices of the ball of radius~$l$ with center~$1$ in the Cayley graph of $\mathrm{PSL}(2,\mathbb Z)$ with generators $A^{\pm 1}$, $B^{\pm 1}$.  Among the elements of~$B(l)$,  we look at the proportion of those which act loxodromically on the Farey graph, and whose axis in the Farey graph passes at distance~1 from the point~$P=0=\frac{0}{1}$. Then the $\liminf_{l\to\infty}$ of this proportion is strictly positive. 
\end{proposition}

\begin{proof}
We have to check that the hypotheses of Theorem~\ref{T:main} are satisfied with~$R=1$ and with geodesic normal form words.
One can check that any element of $\mathrm{PSL}(2,\mathbb Z)$ except $ABA$ and $\overline A\overline B\overline A$ has a unique representative word not containing the subwords $ABA$, $\overline A\overline B\overline A$, $BAB$, $\overline B\overline A \overline B$, $AB\overline A$, $BA\overline B$, $\overline A\overline B A$ and $\overline B\overline A B$. The automaton describing this language is shown in Figure~\ref{F:automaton}. Note that the words in this normal form are in fact geodesics in the Cayley graph.


\begin{figure}[htb] 
\centering 
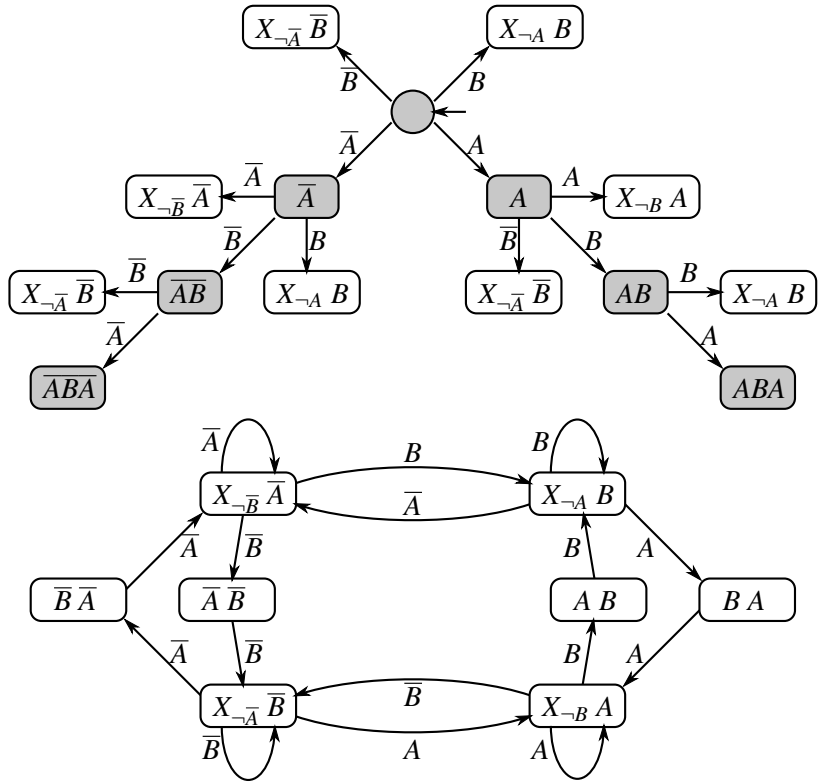
\caption{A finite state automaton for $\mathrm{PSL}(2,\mathbb Z)$. All states (except the start state) are accept states. The states are named after the last two letters read. The symbol $X_{\neg A}$ means ``any letter but $A$, or no letter''. The bottom part of the picture shows the accessible sub-automaton. The top part shows (shaded grey) the non-accessible states, and how they are connected to the accessible states. Not shown here is the fail state and the arrows leading to it.}
\label{F:automaton}\end{figure}


One can then write down the adjacency matrix~$M$ of the accessible sub-automaton -- this is an $8\times 8$ matrix. According to Perron-Frobenius theory, there is a unique eigenvalue $\lambda$ of $M$ of maximal modulus. This~$\lambda$ is then real and has a one-dimensional eigenspace. In the present  case, it turns out that $\lambda=\sqrt{2}+1$ (the ``silver ratio'').


We see that several words can play the role of the word $w_{\rm rigid}$, we choose $w_{\rm rigid}=B$. (Another possible choice would have been $w_{\rm rigid}=A\overline B$.)

In Lemma~\ref{L:RigidNotRare} we saw that the proportion of rigid words among all words of length~$l$ is bounded away from zero. In order to prove this, we showed that even the subset consisting of the words that start with the subword $w_{\rm rigid}$ and trace a path in the automaton that terminates in the same state as~$w_{\rm rigid}$ has non-vanishing frequency. For the present example, one can explicitly calculate this proportion, using the adjacency matrix~$M$. One finds that  among all elements of length~$l$, the proportion of words starting with the letter~$B$ is~$\frac{1}{4}$, and among those, the proportion of words that end in the state $X_{\neg A} B$ tends to $\frac{\lambda}{4\lambda+4}=\frac{1}{4\sqrt{2}}$. This yields a (not at all sharp) lower bound of $\frac{1}{16\sqrt{2}}$ on the limit of the proportion of rigid words.

In order to study the proportion of loxodromically acting words among the rigid ones, we equip the Farey graph with the base point~$P=0=\frac{0}{1}$. Then the geodesic words hypothesis is satisfied with $R=1$, i.e.\ the orbit of~$P$ under the action of a normal form word traces out a path which stays at distance~1 from any geodesic connecting its end-points.  
Now by Claim~B, the normal form of any rigid element which does not act loxodromically cannot contain any subword that moves the base point by more than five. However, interesting loxodromic actions exist -- for instance given by the word $A\overline B$ -- and the action of $(A\overline B)^6$ moves the base point by a distance of 6. The proportion, among rigid (or indeed any) words of length~$l$, of words which do not contain the subword $(A\overline B)^6$ tends to zero as $l$~tends to infinity. 
\end{proof}


\subsection{Word-hyperbolic groups acting on their Cayley graphs}

Let~$G$ be an infinite word-hyperbolic group equipped with some fixed, finite generating set~$\mathcal S$. Let $X$ be the Cayley graph of~$G$ with respect to~$\mathcal S$; it is $\delta$-hyperbolic. As the base point of~$X$ we take~$1_G$, the neutral element of~$G$. The group acts on~$X$, on the left, by isometries. Elements of finite order act elliptically, and elements of infinite order act loxodromically~\cite[Chapter~8]{GhysDeLaHarpe}.

We wish to prove that a positive proportion of elements, in the sense of Theorem~\ref{T:main}, act loxodromically. Unfortunately, it seems hard to apply Theorem~\ref{T:main} directly, as the automatic normal form hypothesis does not hold in sufficient generality. However, in the framework of hyperbolic groups we have other techniques at our disposal which will allow us, in Proposition~\ref{P:MainHyperbolic}, to prove an analogue of Theorem~\ref{T:main}, under a very mild technical condition. 

We consider the language $\mathcal L$ of words with letters in~$\mathcal S^{\pm 1}$ such that every subword of length $10\delta$ is a geodesic in~$X$.  
In other words, the language $\mathcal L$ is defined by the fact that a finite number of subwords, namely the non-geodesics of length at most $10\delta$, are banned.

The technical hypothesis on the pair $(G,\mathcal S)$ that we need is that for any pair of words $(w_1,w_3)$ in $(\mathcal L, \mathcal L)$, there exists a word $w_2$ in~$\mathcal L$ such that the composite word $w_1w_2w_3$ also belongs to~$\mathcal L$. 
(We conjecture that this hypothesis is satisfied for any hyperbolic group with any set of generators.)

\begin{proposition}\label{P:MainHyperbolic}
Suppose $G$ is a group, equipped with a finite generating set~$\mathcal S$, whose Cayley graph~$X$ is $\delta$-hyperbolic. Suppose also that the above technical hypothesis on the pair $(G,\mathcal S)$ is satisfied. Then there exist constants $D,R$ with the following properties: for every element $g$ of~$G$ with $d_X(1_G,g) \geqslant 10\delta$, there exists an element~$\tilde g$ with $d_X(g,\tilde{g})\leqslant D$
which acts loxodromically on~$X$ (i.e.\ is torsion free), and whose loxodromic axis in~$X$ passes within distance~$R$ from the basepoint $1_G$.
\end{proposition}

\begin{proof}
For every pair of words $(w_1,w_3)$, both of length $10\delta$ and belonging to~$\mathcal L$, we choose a word $w_2$ as in the technical hypothesis, where this $w_2$ should be chosen as short as possible. Let $D$ be the length of the longest word $w_2$ thus obtained.

Now for any element $g$ of~$G$ at distance at least $10\delta$ from the identity element. Let $w$ be a geodesic representative of~$g$, let $w_1$ consist of the last $10\delta$ letters of $w$, and let $w_3$ consist of the first $10\delta$ letters. 
With $w_2$ the word promised by the technical hypothesis, let $\tilde g$ be the element represented by the word $w w_2$. We observe that $d_X(g,\tilde g)\leqslant D$, and that any power of the word $w w_2$ representing $\tilde g$ belongs to the language~$\mathcal L$. 

Now we recall a classical fact about hyperbolic group due to Canary -- see \cite[Theorem 3.11]{ABCFLMSS}:  words in~$\mathcal L$ represent $(12,7\delta)$-quasigeodesics in~$X$. 

Let us denote $x(i)$, for $i\in\mathbb N$, the point of the Cayley graph~$X$ represented by the prefix of length~$i$ of the word $w w_2 w w_2 w\ldots$ Similarly, for $i<0$, we let $x(i)$ be the prefix of length~$i$ of $w_2^{-1} w^{-1} w_2^{-1} w^{-1}\ldots$. 
Then the path $(x(i))_{i\in\mathbb Z}$ is a $(12,7\delta)$-quasigeodesic in~$X$ which contains the bi-infinite sequence points $\ldots, \tilde g^{-2}, \tilde g^{-1},1,\tilde g, \tilde g^2,\ldots$, i.e.\ the orbit of the base point of~$X$ under the action of $\langle\tilde g\rangle$. 

Thus the element $\tilde g$ acts loxodromically. Moreover, the axis of this action is a geodesic connecting the same pair of points on~$\partial X$ as the quasigeodesic $(x(i))_{i\in\mathbb Z}$. 
Since in a $\delta$-hyperbolic space, any $(12,7\delta)$-quasigeodesic stays at some universally bounded distance~$R$ from any geodesic connecting the same pair of points \cite[Chapter 5]{GhysDeLaHarpe}\marg{This reference talks about finite geodesics. Reference for the infinite case?}, any axis of~$\tilde g$ passes at distance~$R$ from the base vertex~$1_G$ of~$X$.
\end{proof}



\subsection{Relatively hyperbolic groups}

The geodesic words hypothesis of Theorem~\ref{T:main} is often very difficult to verify -- for instance, we do not know if it holds in the case of braid groups (with the Garside normal form) acting on the curve complex of a punctured disk.  The aim of this short section is to point out that there is one large and interesting class of groups acting on $\delta$-hyperbolic spaces where this hypothesis holds automatically. 

Let $G$ be a group, equipped with a finite generating set~$S$, that is relatively hyperbolic, in the strong sense, with respect to the finite collection of subgroups $\mathcal H=\{H_1,\ldots,H_m\}$. We recall what this means. Firstly, the Cayley graph of~$G$ with respect to~$S$, with every translate of every~$H_i$ coned off, is required to be a $\delta$-hyperbolic space. We denote this quotient space by $Cayley(G,S\cup\mathcal H)$. Secondly, the couple $(G,\mathcal H)$ must satisfy the so-called bounded coset penetration property. (This is Farb's definition~\cite{Farb}, for a list of equivalent definitions see~\cite{GrovesManning}).

\begin{proposition}
Suppose that $G$ is a finitely generated group (generating set~$S$), which is (strongly) relatively hyperbolic, with respect to the collection~$\mathcal H$ of subgroups, so that $G$ acts on the $\delta$-hyperbolic complex $Cayley(G,S\cup\mathcal H)$. Suppose further that~$G$ satisfies the automatic normal form hypothesis, and that the automatic normal forms are \emph{geodesics} in the Cayley graph $Cayley(G,S)$. Then the geodesic words hypothesis is satisfied.
\end{proposition}

\begin{proof} This is an immediate consequence of Proposition~8.25 and Theorem~A.1 in~\cite{DrutuSapir}. \end{proof}


\subsection{Braid groups acting on the additional length complex}

An important class of automatic groups, admitting algorithms for the word- and conjugacy problem which are very fast in practice and relatively easy to program are the so-called \emph{Garside groups}. For a gentle introduction to these groups see~\cite[Section 1]{BGGMrigid}, and for much more complete accounts see~\cite{DehornoyGarside,GarsideFoundations}.
The classical examples of such groups are Artin-Tits groups of spherical type, and in particular braid groups~\cite{BrieskornSaito,BessisDual,ElrifaiMorton,ECHLPT}. 


It was proved in~\cite{CalvezWiest} that every Garside group~$G$ acts on a $\delta$-hyperbolic space $\mathcal C_{AL}(G)$ in such a way that the shadow of words in automatic normal form on $\mathcal C_{AL}(G)$ is an un-parametrized quasi-geodesic. More precisely, it is shown that the geodesic words hypothesis is satisfied with~$R=39$.

In the case $G=B_n$, the braid group on $n$ strands equipped with its classical Garside structure, the complex~$\mathcal C_{AL}(B_n)$ is of infinite diameter. Moreover, periodic and reducible  braids act elliptically, and there exists a positive, rigid pseudo-Anosov braid which acts loxodromically, so there are interesting loxodromic actions. (It is currently unknown whether all pA braids act loxodromically, and indeed whether the additional length complex is quasi-isometric to the curve complex.)

\begin{notation}
We denote $\hat B(l)$ the ball with radius~$l$ and center~1 in the Cayley graph of~$B_n$, with  Garside's set of generators.
\end{notation}

Most of the following result was already known from S.~Caruso's paper~\cite{CarusoGeneric1}: 

\begin{proposition}\label{P:Caruso} 
Among the elements in the ball~$\hat B(l)$, the $\liminf_{l\to\infty}$ of the proportion of elements which act loxodromically on $\mathcal C_{AL}(B_n)$ (and thus are pseudo-Anosov), and whose axis in~$\mathcal C_{AL}(B_n)$ passes at distance at most 39 from the base point~$1_G$, is strictly positive. 
\end{proposition}

\begin{proof}
Let us first restrict our attention to the positive braid monoid~$B_n^+$. As seen in Example~\ref{E:B_n^+}, the automatic normal form hypothesis is satisfied in~$B_n^+$. 
Also, we saw in the preceding discussion that the other hypotheses of Theorem~\ref{T:main} are satisfied. Applying Theorem~\ref{T:main} and Remark~\ref{R:SphereBall}, we obtain the desired result (Proposition~\ref{P:Caruso}), with one exception: we do not obtain a statement about the balls $\hat B(l)$ of radius~$l$, but instead about the ``sphere segment'' of \emph{positive} braids whose Garside normal form is of length \emph{exactly} equal to~$l$. Let us denote this sphere segment by~$S(l)$.

In order to deduce the analogue result for $\hat B(l)$, we recall first that in~$B_n^+$, like in all Garside monoids, words in Garside normal form represent geodesics in the Cayley graph~\cite{CharneyMeier}. More generally, the ball~$\hat B(l)$ consists precisely of braids 
which can be written as a product $\Delta^{-k}\cdot x$, where $k\in \{0,1,\ldots,l\}$ and $x\in S(l)$; moreover, this writing, when it exists, is unique. Now for every braid $x$, the braids $\Delta^{-2} x, \Delta^{-4} x,\ldots$ all have the same Nielsen-Thurston type, and the same action on~$\mathcal C_{AL}(B_n)$, as~$x$. 
Thus if we compare the proportions, in~$\hat B(l)$ and in~$S(l)$, of elements which act loxodromically and with axis passing close to the base point, we see that the proportion in~$\hat B(l)$ is equal to at least half the proportion in~$S(l)$. 
\end{proof}

Actually, we can also give a new proof of the main result from~\cite{CarusoWiestGeneric2}. It is a small, but significant variation of the original proof. The statement of the result is as follows:

\begin{proposition}\label{P:CarusoW}
Among the elements in the ball $\hat B(l)$ 
the proportion of elements which act loxodromically on $\mathcal C_{AL}(B_n)$ tends to 1 exponentially quickly as $l\to\infty$. 
\end{proposition}

\begin{proof} 
The new ingredients that we shall use are the action of $B_n$ on~$\mathcal C_{AL}(B_n)$, and Lemma~\ref{L:NonPaObvious} from the present paper. By contrast, we will avoid using results from~\cite{GonzalezMenesesWiest2}.

Throughout this proof, by ``generic elements'' we mean ``a proportion of elements which tends to 1 exponentially quickly''. As seen in~\cite{CarusoWiestGeneric2}, generic elements of~$B_n^+$ have a so-called \emph{non-intrusive conjugation} to a rigid element. Moreover, the middle fifth of this conjugated element, which has not been touched by the conjugation, contains generically any given braid as a subbraid -- for instance a braid whose action moves the base point of $\mathcal C_{AL}(B_n)$ by more than $5\cdot R=195$. By Lemma~\ref{L:NonPaObvious}, these generic elements act loxodromically on~$\mathcal C_{AL}(B_n)$. This completes the proof of Proposition~\ref{P:CarusoW}.

In other words, we can use literally the same proof as the proof of Theorem~5.1 in~\cite{CarusoWiestGeneric2}, with one exception: Lemma 5.2 of~\cite{CarusoWiestGeneric2} uses the fact that a rigid braid whose normal form contains one of two very particular subwords (called $\gamma_1$ and $\gamma_2$ in~\cite{CarusoGeneric1}) must be pseudo-Anosov. Instead, we now use the fact that a rigid braid whose normal form contains another subword, namely one whose action moves the base point of $\mathcal C_{AL}(B_n)$ by more than~$195$, must be pseudo-Anosov. 
\end{proof}


What is the significant advantage of our new proofs of Propositions~\ref{P:Caruso} and~\ref{P:CarusoW}? The proofs which were given in~\cite{CarusoGeneric1} and~\cite{CarusoWiestGeneric2} both relied heavily on a result of~\cite{GonzalezMenesesWiest2} (Theorem 5.16), which is very specific to braid groups. It states that if a rigid braid is reducible, then its reducibility is obvious (there is a round or ``almost round'' reducing curve). 
In our current proof, this result is no longer used; instead, our Claim~A plays a similar r\^ole. 
(Similarly, our Claim~B can be thought of as an analogue of the theorem of Bernadete-Guttierez-Nitecki~\cite{BGN,CalvezDual}.)
We believe that our new proof of the genericity of pseudo-Anosov braids can be generalized so as to apply to many other groups, e.g.\  to Artin-Tits groups of spherical type. 


\subsection{Fundamental groups of locally-CAT(0) cubical complexes acting on the contact graph}

In this section we consider a locally finite, non-positively curved (or ``locally CAT(0)'') cubical complex~$Y$, equipped with a base point~$b$. We shall look at the fundamental group $\pi_1(Y,b)$. For instance, $Y$ could be the the defining complex, the so-called Salvetti complex, of a right-angled Artin group $A=\pi_1(Y)$, which has one $n$-cube for every $n$-tuple of pairwise commuting generators.

We have already studied this general setup in Example~\ref{E:pi_1(CAT(0))}. There we saw that the associated finite state automaton may or may not be recurrent, but if it is, then the automatic normal form hypothesis is satisfied. 

The universal cover of~$Y$ is a CAT(0) cubical complex~$\widetilde Y$. 
Now  the fundamental group $G=\pi_1(Y)$ acts on~$\widetilde Y$ by deck-transformations, but unfortunately, if we want the fundamental \emph{groupoid} of $Y$, which in Example~\ref{E:pi_1(CAT(0))} was called~$\hat G$, to act on~$\widetilde Y$, we must assume that~$Y$ has only a single vertex. 
In order to circumvent this problem, we use the following trick.

We define $Y'$ to be the cubical complex obtained from the complex~$Y$ by identifying all its vertices. This new complex~$Y'$ is still locally CAT(0), as Gromov's link condition in still satisfied. If~$Y$ has $v$ vertices, then the universal cover $\widetilde {Y'}$ consists of infinitely many copies of~$\widetilde Y$, glued together in a tree-like fashion, so that at each vertex, $v$~copies of~$\widetilde Y$ meet. 
Any choice of a base point~$\tilde b$ of~$\widetilde {Y'}$ covering the vertex of~$Y'$ induces a choice of a distinguished copy of $\widetilde Y$ in~$\widetilde {Y'}$, which we call~$\widetilde Y_*$: it consists of the set of end-points of lifts, starting at~$\tilde b$, of paths in~$Y$ starting at~$b$.

The natural glueing map $Y\to Y'$ induces an \emph{injective} group homomorphism $\pi_1(Y)\hookrightarrow \pi_1(Y')$. It even induces an injective groupoid homomorphism from the fundamental groupoid~$\hat G$ of~$Y$ to the fundamental group $\pi_1(Y')$. 
Therefore, we obtain an action of the groupoid~$\hat G$ on the universal cover $\widetilde {Y'}$ of~$Y'$ by deck-transformations. 
The action of an element~$g\in\hat G$ on~$\widetilde {Y'}$ preserves the distinguished copy $\widetilde Y_*$ of~$\widetilde Y$ if and only if $g\in G$, i.e. if the path in~$Y$ representing~$g$ terminates in the base point~$b$, not some other vertex. Nevertheless, for any $g\in \hat G$, the vertex $g.\tilde b$ of~$\widetilde {Y'}$ lies in $\widetilde Y_*$.

Now we recall some definitions and  results due to Hagen~\cite{Hagen13,Hagen14}, see also~\cite{BehrstockHagenSisto} for a nice overview. The \emph{contact graph} $\mathcal C\widetilde Y$ of~$\widetilde Y$ is the graph having one vertex for every hyperplane of~$\widetilde Y$, and two vertices are connected by an edge if the corresponding hyperplanes are not separated by a third hyperplane (in other words, if they intersect or osculate). It was proven by Hagen that the contact graph is always connected and $\delta$-hyperbolic -- in fact, it is quasi-isometric to a tree. In particular, there are no parabolic isometries of~$\mathcal C\widetilde Y$, only elliptic and loxodromic ones. (See \cite[Theorem~2.8]{BehrstockHagenSisto} for more detailed descriptions of each of these two cases.)

Similarly, we have a contact graph~$\mathcal C\widetilde {Y'}$, and the distinguished copy $\widetilde Y_*$ of $\widetilde Y$ in $\widetilde {Y'}$ gives rise to a distinguished copy $\mathcal C\widetilde Y_*$ of $\mathcal C\widetilde Y$ in $\mathcal C\widetilde {Y'}$. The fundamental groupoid~$\hat G$ of~$Y$ acts on~$\mathcal C\widetilde {Y'}$ by graph isomorphisms, and an element $g$ of~$\hat G$ preserves the subset~$\mathcal C\widetilde Y$ if and only if~$g\in G=\pi_1(Y)$. 

\begin{proposition}
Suppose $Y$ is a non-positively curved cubical complex such that, firstly, the associated finite-state automaton is recurrent, and secondly, the contact graph~$\mathcal C\widetilde Y$ is of infinite diameter. Let~$B(l)$ be the set of elements of~$\pi_1(Y)$ which can be written as a product of at most~$l$ elements of~$\mathcal S$ (at most~$l$ cube diagonals). Then the $\liminf_{l\to\infty}$ of the proportion of elements of~$B(l)$ which act loxodromically on~$\mathcal C\widetilde Y$ and whose axis passes at distance $R=3$ from the base point  is strictly positive.
\end{proposition}

\begin{proof}
The action of an element $g\in\pi_1(Y)$ on~$\mathcal C\widetilde {Y'}$ is loxodromic if and only if its action on~$\mathcal C\widetilde Y$ is loxodromic. Thus for the rest of the proof we shall study the action of~$\pi_1(Y)$ on~$\mathcal C\widetilde {Y'}$ (rather than on~$\mathcal C\widetilde Y$), and the proportion of elements which act loxodromically.
We have already seen in Example~\ref{E:pi_1(CAT(0))} that the automatic normal form hypothesis of Theorem~\ref{T:main} is verified. We still have to prove the geodesic words hypothesis, and the hypothesis that interesting loxodromic actions exist.

First we define a base point in~$\mathcal C\widetilde {Y'}$. In fact, we shall not be very interested in choosing a base \emph{point}, but a base \emph{simplex} $\mathcal P$ in~$\mathcal C\widetilde {Y'}$, i.e.\ a finite collection of points which are all at distance one from each other. By definition, the vertices of our base simplex are those corresponding to all hyperplanes in~$\widetilde {Y'}$ that are adjacent to the base point~$\tilde b$. (As a base point of~$\mathcal C\widetilde Y$ one can then choose any vertex of~$\mathcal P$ belonging to~$\mathcal C\widetilde Y$.)

Now let~$\gamma$ be an element of~$\pi_1(Y)$, and let $\gamma=\gamma_1\cdot\ldots\cdot\gamma_k$ be a shortest possible decomposition of~$\gamma$ as a product of our generators~$\mathcal S$ of~$\hat G$ -- for instance the normal form of~$\gamma$. 
The trace of the base simplex~$\mathcal P$ under the normal form of~$\gamma$, i.e.\ the sequence of simplices $1.\mathcal P$, $\gamma_1.\mathcal P$, $\gamma_1\gamma_2.\mathcal P,\ldots,\gamma.\mathcal P$, has the following property: 
each of the simplices contains at least one vertex\marg{explain more?} which represents a hyperplane separating $\tilde b$ and~$\gamma.\tilde b$, and conversely every such hyperplane appears as a vertex of at least one of the simplices. 
Thus the trace of the base simplex under the action of $\gamma_1 \gamma_2\gamma_3\ldots$ and the set of hyperplanes separating $\tilde b$ and $\gamma.\tilde b$ are at Hausdorff distance~1 in~$\mathcal C\widetilde {Y'}$. 

This implies that, if $\gamma=\gamma_1\cdot\ldots\cdot\gamma_k$ and $\gamma=\gamma'_1\cdot\ldots\cdot \gamma'_k$ are two different representatives of the same minimal length, then the traces of the base simplex under these two words are at Hausdorff distance at most~2 in~$\mathcal C\widetilde {Y'}$. 
However, according to Theorem~A(2) of~\cite{BehrstockHagenSisto}, there is such a word for which the trace of the base simplex is at Hausdorff distance~1 from a \emph{geodesic} in~$\mathcal C\widetilde {Y'}$. 
In summary, the trace of the base simplex under the action of the normal form word representing~$\gamma$ stays at Hausdorff distance~3 from a geodesic in~$\mathcal C \widetilde {Y'}$ between points of~$\mathcal P$ and $\gamma.\mathcal P$.

By exactly the same argument, we have more generally that the sequence of simplices $\gamma_1\cdot\ldots\cdot \gamma_i.\mathcal P$, $\ldots$, $\gamma_1\cdot\ldots\cdot\gamma_j.\mathcal P$ in~$\mathcal C \widetilde {Y'}$ (for $1\leqslant i<j\leqslant k$) is uniformly close to a geodesic between points of $\gamma_1\cdot\ldots\cdot \gamma_i.\mathcal P$ and $\gamma_1\cdot\ldots\cdot\gamma_j.\mathcal P$.
We can conclude that the geodesic words hypothesis is satisfied with~$R=3$. 

Now we need to check the hypothesis that interesting loxodromic actions exist. By Lemma~\ref{L:NonPaObvious} it suffices to construct a word with letters in~$\mathcal S$ which 
is read along 
a \emph{loop} in the finite state automaton 
and whose action moves the base-simplex~$\mathcal P$ by 
more than $5\cdot R$, i.e.\ by more than 15.

Here is a way to perform this construction: since the automaton is recurrent, we can choose a finite family of words over the alphabet~$\mathcal S$, one for every ordered pair of states, tracing out a path in the FSA from the first state to the second. 
Since we chose only finitely many such words, there exists some~$\delta>0$ such that all the words move~$\mathcal P$ by at most~$\delta$. 
Now take an element~$\gamma$ of~$\pi_1(Y)$ whose action action moves~$\mathcal P$ by at least~$5\cdot R+\delta$ -- this exists, since~$\mathcal C\widetilde{Y'}$ is of infinite diameter. Let $w_1$ be the normal form word representing~$\gamma$. 
Let $S_1$ be the state of the automaton reached from the start state after reading the first letter of~$w_1$, and let $S_2$ be the state attained after reading the whole word~$w_1$. Let $w_2$ be the word from our finite family that allows us to move from $S_2$ to~$S_1$. Finally, let $w$ be the word obtained by removing the first letter from the concatenated word $w_1w_2$. This word~$w$ traces out a loop in the automaton which passes through the states $S_1$ and~$S_2$, and it still moves the base simplex by at least~$5\cdot R$. 
\end{proof}

{\bf Acknowledgements } I thank the following people for helpful remarks and conversations: Sandrine Caruso, Matthieu Calvez, Vincent Guirardel, Alessandro Sisto, Camille Horbez, Saul Schleimer, Thierry Coulbois, Lee Mosher, Alexandre Martin, Hamish Short, Serge Cantat, Christophe Dupont.
This research was supported by the grant ANR LAM (ANR-10-JCJC-0110).


\end{document}